\documentclass[reqno,twoside, 11pt]{amsart}
\usepackage{amsfonts,xcolor,eucal,upgreek}
\usepackage{a4wide,amsmath,amssymb,latexsym,amsthm}
\usepackage[dvips,bottom=1.4in,right=1in,top=1in, left=1in]{geometry}
\usepackage{graphicx,epstopdf,xspace,enumerate,pstricks}
\usepackage{mathrsfs,mathtools,epic,bm,commath,dsfont}
\usepackage[colorlinks=false,urlcolor=blue,
citecolor=red,linkcolor=blue,linktocpage,pdfpagelabels,
pagebackref=true,
bookmarksnumbered,bookmarksopen]{hyperref}
\usepackage[english]{babel}

\newfont{\script}{eusm10 scaled\magstep1}

\newcommand{\R}{\mathbb{R}}

\newcommand{\dt}{\mathrm{d}t}

\renewcommand{\d}{\mathrm{d}}
\newcommand{\dx}{\mathrm{d}x}
\newcommand{\dy}{\mathrm{d}y}
\renewcommand{\upxi}{\mathcal{E}}
\renewcommand{\star}{*}

\newcommand{\ie}{i.e.,\@\xspace}
\def\Omc{\mathbb{R}^N\setminus\Omega}

\def\RR{{\mathbb{R}}}

\def\Om{\Omega}
\def \dis {\displaystyle}
\usepackage{ifthen}
\usepackage{xargs}
\newcommandx{\yaHelper}[2][1=\empty]{%
\ifthenelse{\equal{#1}{\empty}}%
  { \ensuremath{ \scriptstyle{ #2 } } } 
  { \raisebox{ #1 }[0pt][0pt]{ \ensuremath{ \scriptstyle{ #2 } } } }  
}   
\newcommandx{\yrightarrow}[4][1=\empty, 2=\empty, 4=\empty, usedefault=@]{%
\ifthenelse{\equal{#2}{\empty}}
  { \xrightarrow{ \protect{ \yaHelper[ #4 ]{ #3 } } } } 
  { \xrightarrow[ \protect{ \yaHelper[ #2 ]{ #1 } } ]{ \protect{ \yaHelper[ #4 ]{ #3 } } } } 
}

\numberwithin{equation}{section}

\usepackage{xcolor}
\usepackage{soul}
\definecolor{jddcol}{rgb}{0,0,0.8}
\definecolor{Warcol}{rgb}{1,0,0}
\definecolor{GMcol}{rgb}{0,1,0.4}

\definecolor{GF}{rgb}{0,0,1}
\definecolor{M}{rgb}{0.1,0.2,0.7}
\definecolor{H}{rgb}{0.7,0.1,0.2}
\definecolor{Y}{rgb}{0.8,0.35,0.1}
\definecolor{MY}{rgb}{0.5,0,0.45}

\newtheorem{theorem}{Theorem}[section]
\newtheorem{proposition}[theorem]{Proposition}
\newtheorem{lemma}[theorem]{Lemma}

\theoremstyle{definition}
\newtheorem{definition}[theorem]{Definition}

\newtheorem{remark}[theorem]{Remark}
\newtheorem{assumption}[theorem]{Assumption}

\title[Nonlocal complement value problem]{Nonlocal complement value problem for a global in time parabolic equation}

\author[Djida]{Jean-Daniel Djida}
\author[Foghem]{Guy Fabrice Foghem Gounoue}
\author[Tchaptchie]{Yannick Kouakep Tchaptchie}
\address[Djida]{{\small African Institute for Mathematical Sciences (AIMS), P.O. Box 608, Limbe Crystal Gardens,South West Region, Cameroon. Email: jeandaniel.djida@aims-cameroon.org}}
\address[Foghem]{{\small Fakult\"{a}t f\"{u}r Mathematik Institut f\"{u}r Analysis, TU Dresden Zellescher Weg 23/25, 01217, Dresden, Germany. Email: guy.foghem@tu-dresden.de}}
\address[Tchaptchie]{{\small Department of SFTI-School of Chemical Engineering and Mineral Industries(EGCIM), University of N'Gaound\'{e}r\'{e}, P.O.Box 454 N'Gaound\'{e}r\'{e} (Cameroon). Email: kouakep@aims-senegal.org}}
\thanks{The first author is supported by the Deutscher Akademischer Austausch Dienst/German Academic Exchange Service (DAAD)}
\thanks{The second author is supported by the Deutsche Forschungsgemeinschaft/German Research Foundation (DFG) via the Research Group 3013: “Vector-and Tensor-Valued Surface PDEs”}

\thanks{The third author is supported  by AIMS-Cameroon research center
	via the travel grant research project: Mission de recherche/EGCIM.}



\begin{document}
\begin{abstract}
The overreaching goal of this paper is to investigate the existence and uniqueness of weak solution of a semilinear parabolic equation with double nonlocality in space and in time variables that naturally arises while modeling a biological nano-sensor in the chaotic dynamics of a polymer chain. In fact, the problem under consideration involves a symmetric integrodifferential operator of L\'{e}vy type and a term called the interaction potential, that depends on the time-integral of the solution over the entire interval of solving the problem. Owing to the Galerkin approximation, the existence and uniqueness of a weak solution of the nonlocal complement value problem is proven for small time under fair conditions on the interaction potential.
\end{abstract}
\maketitle

\section{Introduction}
Let $\Omega$ be an open bounded set of $\mathbb{R}^N(N \geq 1)$. 
For $T>0$, we are interested in studying  the following nonlocal complement value problem
\begin{align}\label{eq:main-equation}
\begin{cases}
\partial_t u +\mathscr{L} u + \displaystyle \varphi \Big(\int_0^T u(\cdot,\tau)\,\mathrm{d}\tau\Big)\, u =0 &\mbox{ in }\; \Omega_T \coloneqq \Omega\times (0,T),\\
u = 0 &\mbox{ in }\; \Sigma \coloneqq (\Omc)\times (0,T), \\
u(\cdot,0) = u_{0} &\mbox{ in }\; \Omega,
\end{cases}
\end{align}
where $u=u(x,t)$ is an unknown scalar function, and $\varphi$ a scalar function that will be specified in the sequel. The initial state $u_0:\Omega\to \mathbb{R}$ is prescribed. Here, we restrict ourselves to a purely integrodifferential operator of L\'{e}vy type $\mathscr{L}$, which is a particular type of nonlocal operators acting on a measurable function $u: \RR^N \to \RR$ as follows
\begin{equation}\label{eq:Levy-operator}
\mathscr{L}u(x) \coloneqq \mathrm{p.v.} \int_{\RR^N}(u(x)-u(y))\nu(x-y)\;\dy, \qquad (x \in \RR^N),
\end{equation}
whenever the right hand side exists and makes sense. Here and henceforward, the function $\nu: \RR^N \setminus \{0\} \to [0,\infty)$ is the density of a symmetric L\'{e}vy measure. In other words, $\nu\geq0$ and measurable such that
\begin{equation*}
\nu(-h) = \nu(h) \quad \mbox{for all} \quad h \in \RR^N \quad \mbox{and} \quad \int_{\RR^N}(1 \wedge |h|^{2})\nu(h)\;\mathrm{d}h <\infty.
\end{equation*}
Notationally, we write $a \wedge b$ to denote $\min(a,b)$ for $a, b \in \RR$. For the sake of simplicity,  we also assume that $\nu$ does not vanish on sets of positive measure. To wit, $\nu$ is fully supported on $\R^N.$ 
A prototypical example of an operator $\mathscr{L}$ is the fractional Laplacian $(-\Delta)^s$, which is obtained by taking $\nu(h) = C_{N,s}|h|^{-N-2s}$ for $h \neq 0$ where $s \in (0,1)$ is fixed and the constant $C_{N,s}$ is given by 
\begin{align*}
C_{N,s} 
=\frac{2^{2s}s\Gamma\big(\frac{N+2s}{2}\big)}{\pi^{N/2}\Gamma(1-s)}.
\end{align*}
The constant $C_{N,s}$ is chosen so that the Fourier relation $\widehat{(-\Delta)^{s}}u(\xi) = |\xi|^{2s}\widehat{u}(\xi)$, $\xi\in \RR^N$, holds for all $u \in C^{\infty}_{c}(\RR^N)$. The fractional Laplacian is one of the most heavily studied integrodifferential operators; see for instance \cite{BV16,CS07,Kw17,Hitchhiker,APM20,Ga19,WYP20} for some basics. Additional results related to the fractional Laplacian can be found, in the references \cite{War,AW20,Caf1,DNA2019}. The operator $\mathscr{L}$ in \eqref{eq:Levy-operator} arises naturally in probability theory as the generator of pure L\'evy stochastic processes with jump interaction measure $\nu(h)\d h$. We refer interested readers to \cite{Sat13,App09,Ber96} for more details on L\'evy processes. Recent studies of Integro-Differential Equations(IDEs) involving nonlocal operators of the form $\mathscr{L}$ in \eqref{eq:Levy-operator} can be found in  \cite{Foghem,FK22}. There exists a substantial amount of literature on nonlocal problems involving L\'evy type operators. For example, see \cite{FKV15,Ru18,FGKV20,FK22} for the study of elliptic problems, see \cite{DyKa20,COS17,Ros15,Ros14} for the regularity of elliptic problems, \cite{FK13,KS13,FR17} for the regularity of parabolic. 

Our main result (see Theorem \ref{thm:main-result}) consists in proving the existence and uniqueness of a weak solution to the problem \eqref{eq:main-equation} for $T$ sufficiently small by imposing some conditions on the initial value $u_0$. It should be noted that we do not require any comparability between the operators $\mathscr{L}$ and $(-\Delta)^s$;  in the sense that the kernels $\nu$ and $|\cdot|^{-N-2s}$ need not be comparable.  Following \cite{St21}, we do this provided that  the potential  $\varphi$ satisfies the following assumption which admits functions $\varphi$ that are not convex and not increasing at $\infty$.
\begin{assumption}\label{eq:assumption-phi}
The potential $\varphi:\mathbb{R}\to [0,\infty)$ is a continuous non-negative function such that $\varphi(0)=0$ and $\tau\mapsto \varphi(\tau)\tau$ is a non-decreasing differentiable function whose derivative is bounded on every compact subset of $\mathbb{R}$.
\end{assumption}

An interesting feature of the problem under consideration is that main equation in \eqref{eq:main-equation} contains a nonlocal operator of L\'{e}vy type in space and a nonlocal in time term that depends on the integral over the whole interval $(0,T)$ on which the problem is being solved, viz., problem \eqref{eq:main-equation} has a global memory, i.e., its depends upon the memory and the future.
Note in passing that if $t\in (0, T)$ is the current time then the memory is recorded in $(0,t)$, while the future is recorded in $(t,T)$. The mixture of nonlocal terms (spatial and time variables) appearing in \eqref{eq:main-equation} renders the problem somehow fully nonlocal with global memory. For this reason, the problem \eqref{eq:main-equation} is termed nonlocal and global in time. It is noteworthy emphasizing that, similar analysis has been carried out in \cite{St21} where the Laplace operator $-\Delta$ is used in place of the nonlocal operator $\mathscr{L}$. There are several works in the literature that study parabolic problems with memory which include the integral of the solution from the initial to the current time, e.g., see \cite{CS11,VDS15,JL08,ZH20}. To the best of our knowledge, the existing problems with memory in the literature differ from ours. Indeed, on the one side, in our problem, we have to deal with the nonlocality in time variable occurred by the semilinear factor depending on the integral over the whole interval $(0,T)$ appearing in equation \eqref{eq:main-equation} whose knowledge  demands to know the so called``future'' which is $T$. Moreover, the nonlocality in time in the system  \eqref{eq:main-equation} is governed by a semilinearity due to the potential $\varphi$. On the other side, we have to deal with the nonlocality in spatial variables, due to the nonlocal L\'evy operator $\mathscr{L}$. The main novelty of the problem \eqref{eq:main-equation} is marked by this double nonlocality in spatial and in time variables rendering the latter the problem  somewhat challenging and of particular interest in its own.
 It is worth emphasizing that the nonlocal problem in \eqref{eq:main-equation}  as well as its local analogue in space cannot be reduced to known ones by any transformation. We refer the interested reader to \cite{Pa95,Sh91}) for the studies of some local problems with initial boundary values  with memory where the ``future'' appears in the data. On the other hand, we point out that our problem arises in the local setting while modeling a biological nanosensor in the chaotic dynamics of a polymer chain s also called, a polymer chain in an aqueous solution. In this model, the density of the probability that a chain occupies certain region of the underlying space allows to describe the position of the chain segment. According to \cite{SS17}, the probability density satisfies with a high accuracy, a certain parabolic equation of the form \eqref{eq:main-equation} in which there is a term responsible for the interaction of the chain of polymer segments.  For more on the application of the problem \eqref{eq:main-equation}, we refer  interested reader for instance to  \cite{St21,St18,SS17} and references therein. 
 Theoretical motivations of studying the nonlocal model \eqref{eq:main-equation} are at least twofold. Firstly, in contrast to the local model, the advantage of the nonlocal model \eqref{eq:main-equation} is that, it allows both smooth and non-smooth solutions $u$ in space variables. Secondly, the nonlocal model can be viewed as an approximation of local model. For example, given $u\in C^2(\R^N)\cap L^\infty(\R^N)$ one can show that $(-\Delta)^s u(x)\to -\Delta u(x)$ as $s\to1^-$. More generally, we have $\mathscr{L}_\varepsilon u(x)\to -\Delta u(x)$ as $\varepsilon\to0$ (see  \cite{Fog21r, FK22}), where $\mathscr{L}_\varepsilon u$ is defined as in \eqref{eq:Levy-operator} with $\nu$ replaced by
 	$\nu_\varepsilon$ satisfying 
 	\begin{align}\label{eq:assumption-nu-alpha}
 	\begin{split}
 	\nu_\varepsilon\geq 0\,\,\text{ is radial}, \quad \int_{\R^N}	(1\land |h|^2)\nu_\varepsilon (h)\d h=\tfrac 1N, \quad \text{and $\forall\,\, \delta>0$\,\,\,} \lim_{\varepsilon\to 0}\int_{|h|>\delta}	\nu_\varepsilon(h)\d h=0\,.
 	\end{split}
 	\end{align} 
 	\noindent If we assume that $\nu\geq 0$ is radial and satisfies $\int_{\R^N}(1\land |h|^2)\nu(h)\d h=\frac1N$ then a remarkable example of family $(\nu_\varepsilon )_\varepsilon$ satisfying \eqref{eq:assumption-nu-alpha}, is obtained from the rescaled version $\nu$ as follows
 	\begin{align*}
 	\begin{split}
 	\nu_\varepsilon(h) = 
 	\begin{cases}
 	\varepsilon^{-N-2}\nu\big(h/\varepsilon\big)& \text{if}~~|h|\leq \varepsilon\\
 	\varepsilon^{-N}|h|^{-2}\nu\big(h/\varepsilon\big)& \text{if}~~\varepsilon<|h|\leq 1\\
 	\varepsilon^{-N}\nu\big(h/\varepsilon\big)& \text{if}~~|h|>1.
 	\end{cases}
 	\end{split}
 	\end{align*}  

 In \cite{St18}, the weak solvability of the problem is proven for the case where $u$ is a positive bounded function and $\varphi$ is the so called Flory-Huggins potential, i.e., is a convex increasing function that tends to infinity as its argument approaches a certain positive value.  The positiveness is a natural requirement since $u$ is a density probability. The landmark works in \cite{St21,St18} demonstrate that, in the case where only the Laplace operator is involved, the problem \eqref{eq:main-equation} is well-posed for sufficiently small time $T$. As a matter of interest the present work takes the result in \cite{St21} to the next stage by using the generator of a pure jump stochastic process of L\'evy type, which is a symmetric nonlocal operator of the form $\mathscr{L}$, to prove further results on weak solvability for this type of problem. The rest of the paper is structured as follows. In Section \ref{Sec:preli}, we provide some well-known results  and functions spaces which are  useful in this paper.  In Section \ref{Sec:auxiliaries}, we prove auxiliary results which are the milestones to prove our core result. Finally, Section~\ref{Sec:WSU} is devoted to the proof of the existence and uniqueness of a weak solution to the problem \eqref{eq:main-equation} thereby constituting the main goal of this article. We prove the existence with the aid of the Tychonoff fixed-point theorem and prove the uniqueness for sufficiently small $T$.

\vspace{2mm}
\emph{Acknowledgment:} The authors thank Victor Starovoitov for helpful and productive   discussions on the proof  Theorem \ref{thm:existence-elliptic} and  Lemma \ref{lem:max-principle}. 

\section{Notations and Preliminaries}\label{Sec:preli}
The purpose of this section is to introduce notations and some preliminary results. Let us collect some basics on nonlocal Sobolev-like spaces in the $L^2$ setting that are generalizations of Sobolev--Slobodeckij spaces and which will be very helpful in the sequel. Let us emphasize that, these function spaces are tailor made for the study of complement value problems involving  symmetric L\'{e}vy operators of type $\mathscr{L}$. We refer the reader to \cite{Foghem} more extensive discussions on this topic.

From now on, unless otherwise stated, $\Omega \subset \RR^N $ is an open bounded set. We also assume that $\nu: \RR^N\setminus\{0\} \to [0,\infty]$ has full support, satisfies the L\'{e}vy integrability condition, \ie $\nu \in L^{1}(\RR^N,(1 \wedge |h|^{2})\mathrm{d}h)$ and is symmetric, \ie $\nu(h) = \nu(-h)$ for all $h \in \RR^N$. We define the space
\begin{equation}\label{NLLSS}
V_{\nu}(\Omega\vert \RR^N) \coloneqq \left\lbrace  u: \RR^N \to \RR \quad \mbox{meas}: \upxi(u,u) < \infty \right\rbrace,
\end{equation}
where $\upxi(\cdot,\cdot)$ is the bilinear form defined by
\begin{equation*}
\upxi(u,v) \coloneqq  \frac{1}{2}\iint\limits_{\mathcal{Q}(\Omega)}(u(x) - u(y))(v(x) - v(y))\nu(x-y)\;\mathrm{d}y\; \d x\,, 
\end{equation*}
where $\mathcal{Q}(\Omega)$ is the cross-shaped set on $\Omega$ given by
\[
\mathcal{Q}(\Omega) \coloneqq (\Omega \times \Omega) \cup (\Omega \times (\RR^N \setminus \Omega)) \cup ((\RR^N \setminus \Omega) \times \Omega).
\] 

\noindent We endow the space $V_{\nu}(\Omega\vert \RR^N)$ with the norm 
\[
\|u\|_{V_{\nu}(\Omega\vert \RR^N)} \coloneqq \Big( \int_{\Omega}|u(x)|^{2}\; \dx +\upxi(u,u) \Big)^{\frac{1}{2}}.
\]

\noindent In order to study the Dirichlet problem \eqref{eq:main-equation} we also need to define the subspace of functions in $V_{\nu}(\Omega\vert \RR^N)$ that vanishes on the complement of $\Omega$, \ie
\[
\mathbb{X}_{\nu}(\Omega \vert \RR^N):= \left\lbrace u \in V_{\nu}(\Omega\vert \RR^N): u = 0 \quad \mbox{a.e.} \quad \mbox{on}\quad \RR^N \setminus \Omega \right\rbrace,
\]
where $V_{\nu}(\Omega\vert \RR^N)$ is defined as in \eqref{NLLSS}. The space $\mathbb{X}_{\nu}(\Omega \vert \RR^N)$ is clearly a closed subspace of $V_{\nu}(\Omega \vert \RR^N)$. Furthermore, if $\partial \Omega$ is continuous then \cite{Foghem} smooth functions of compact support $C_c^\infty(\Omega)$ are dense in $\mathbb{X}_{\nu}(\Omega \vert \RR^N)$. In addition, we have that
\[
\|u\|_{\mathbb{X}_{\nu}(\Omega \vert \RR^N)} = \Big( \iint\limits_{\RR^N\RR^N}(u(x) - u(y))^{2}\nu(x-y)\;\mathrm{d}y\; \mathrm{d}x \Big)^{\frac{1}{2}}
\]
defines an equivalent norm on $\mathbb{X}_{\nu}(\Omega \vert \RR^N)$. Indeed, in virtue of the Poincar\'{e}-Friedrichs inequality on $\mathbb{X}_{\nu}(\Omega \vert \RR^N)$,  there exists a constant $C = C(N,\Omega, \nu)>0 $ depending only on $N,\Omega$ and $\nu$ such that
\begin{equation} \label{eq:poincareX}
\|u\|_{L^2(\Omega)}^{2} \leq C \|u\|^{2}_{\mathbb{X}_{\nu}(\Omega \vert \RR^N)} \quad
\mbox{for every} \quad u\in \mathbb{X}_{\nu}(\Omega \vert \RR^N).
\end{equation}
	
\noindent This can be  verified by observing that $\mathbb{R}^N\setminus B_R(x)\subset \Omc$ for all $x\in \Omega$, where  $R>0$ is the diameter of $\Omega$.   For $u\in \mathbb{X}_{\nu}(\Omega \vert \RR^N)$, we recall that $u=0~a.e$ on $\Omc$. Hence we have  
\begin{align*}
 \|u\|^{2}_{\mathbb{X}_{\nu}(\Omega \vert \RR^N)} &\geq 2\int_{\Omega}|u(x)|^2\d x\int_{\Omc}\hspace{-2ex} \nu(x-y)\d y\\
 &\geq 2\int_{\Omega}|u(x)|^2\d x\int_{\RR^N\setminus B_R(x)} \hspace{-4ex}\nu(x-y)\d y= 2\|\nu_R\|_{L^1(\mathbb{R}^N)}\|u\|^2_{L^2(\Omega)}.
\end{align*}
It suffices to take C= $(2\|\nu_R\|_{L^1(\RR^N)})^{-1}$ with $\nu_R=\nu \mathds{1}_{\mathbb{R}^N\setminus B_R(0)}$. According to \cite{JW19},  the Poincar\'{e}-Friedrichs inequality \eqref{eq:poincareX} remains true  if  $\Omega$ is only bounded in one direction.\medskip

\begin{remark}\label{rmk:equivalence}
	One can observe that the notation $V_{\nu}(\Omega\vert \RR^N)$ is to emphasize that the integral of the measurable map $(x,y) \mapsto (u(x)-u(y))^{2}\nu(x-y)$ performed over $\Omega \times \RR^N$ is finite. Moreover, the latter is equivalent to that performed over $\mathcal{Q}(\Omega), $ that is, $\mathcal{E}(u,u)\asymp\iint_{\mathcal{Q}(\Omega)} (u(x)-u(y))^{2}\nu(x-y)\dy\dx.$ From the local scenario point of view, it is fair to see the space $V_{\nu}(\Omega\vert \RR^N)$ as the nonlocal replacement of the classical Sobolev space $H^1(\Omega)$, whereas $\mathbb{X}_{\nu}(\Omega\vert \RR^N)$ can be viewed as the nonlocal replacement of the classical Sobolev space $H^1_0(\Omega)$.
\end{remark}
The aforementioned spaces are Hilbert spaces. Additional, recent finds about these function spaces and their relations with classical Sobolev spaces can be found in \cite{Foghem, Fog21s,Fog21r}. Let $(V_{\nu}(\Omega \vert \RR^N))^{\star}$ and $(\mathbb{X}_{\nu}(\Omega \vert \RR^N))^{\star}$ be the dual spaces of $V_{\nu}(\Omega \vert \RR^N)$ and $\mathbb{X}_{\nu}(\Omega \vert \RR^N)$ respectively. We have the following continuous Gelfand triple embeddings
\begin{equation*}
\mathbb{X}_{\nu}(\Omega \vert \RR^N) \hookrightarrow L^2(\Omega)\hookrightarrow (\mathbb{X}_{\nu}(\Omega \vert \RR^N))^{\star} \quad \mbox{and}\quad V_{\nu}(\Omega \vert \RR^N) \hookrightarrow L^2(\Omega)\hookrightarrow (V_{\nu}(\Omega \vert \RR^N))^{\star}  .
\end{equation*}
The next result borrowed from \cite{Foghem,JW19,DMT18}, provides  sufficient conditions under which the spaces $\mathbb{X}_{\nu}(\Omega \vert \RR^N)$ and $V_{\nu}(\Omega \vert \RR^N)$ are compactly embedded in $L^2(\Omega)$. 
\begin{theorem}\label{thm:compactness}
	Assume that $\nu \in L^1(\R^N, 1\land |h|^2)$ and  $\nu \not\in L^1(\R^N)$. If $\Omega\subset \R^N$ is open and bounded then the embedding $\mathbb{X}_{\nu}(\Omega \vert \RR^N) \hookrightarrow L^2(\Omega)$ is compact. Furthermore, the embedding $V_{\nu}(\Omega \vert \RR^N) \hookrightarrow L^2(\Omega)$ is also compact provided that in addition, $\Omega$ has a Lipschitz boundary and $\nu$ satisfies
	\begin{align}\label{eq:asymp}
	\lim_{\delta\to 0}\frac{1}{\delta^2} \int_{B_\delta(0)} |h|^2\nu(h)\d h=\infty.
	\end{align} 
\end{theorem}

\noindent It is worthwhile noticing that we have the natural continuous and dense embeddings
\begin{align*}
L^2(0,T;\mathbb{X}_{\nu}(\Omega \vert \RR^N))\hookrightarrow L^2(0,T;L^2(\Omega))\hookrightarrow L^2(0,T;(\mathbb{X}_{\nu}(\Omega \vert \RR^N))^{\star}).
\end{align*}
Next we introduce the Sobolev type space $H_\nu(0,T):=H^1(0,T; \mathbb{X}_{\nu}(\Omega \vert \RR^N)$ by 
\begin{equation*}
H_\nu(0,T)= \Big\{\psi \in L^2(0,T;\mathbb{X}_{\nu}(\Omega \vert \RR^N)):\quad \partial_t\psi \in L^2(0,T;((\mathbb{X}_{\nu}(\Omega \vert \RR^N))^{\star})\Big\}.
\end{equation*}
The space $H_\nu(0,T)$ (see \cite{Tommaso}) is a Hilbert space endowed with the  norm given by
\begin{equation}\label{normW0T}
\|\psi\|^2_{H_\nu(0,T)}=\|\psi\|^2_{L^2(0,T;\mathbb{X}_{\nu}(\Omega \vert \RR^N))}+\|\partial_t\psi_{L^2(0,T;(\mathbb{X}_{\nu}(\Omega \vert \RR^N))^{\star})}.
\end{equation}

\noindent  As a consequence of Theorem \ref{thm:compactness}, we get the  compact embedding \eqref{Lions-compactness} below.

\begin{proposition} \label{prop:emb} Assume that $\nu \not\in L^1(\R^N)$ and  $\Omega\subset \R^N$ is open bounded. The following are true. 
 \begin{enumerate}[$(i)$]
 \item Lions-Magenes Lemma \cite[Theorem II.5.12]{BF13}$:$ the following  embedding  is continuous 
 \begin{equation}\label{contWTA}
 H_\nu(0,T)\hookrightarrow C([0,T];L^2(\Omega)).
 \end{equation}
 
 \item  Lions-Aubin Lemma \cite[Theorem II.5.16]{BF13}$:$ 
 the following embedding  is compact 
  \begin{equation}\label{Lions-compactness}
  H_\nu(0,T) \hookrightarrow L^2(0,T; L^2(\Omega)).
  	\end{equation}
 \end{enumerate} 
\end{proposition}

\medskip

\noindent Now we state the integration by parts formula contained in \cite{Foghem, DROV17} for smooth functions. Precisely for every $\phi, \psi \in C_c^\infty(\R^N)$ following nonlocal Gauss-Green formula holds true
\begin{align}\label{Int-Part}
\upxi(\phi,\psi)=
\int_{\Omega}\psi\mathscr{L}\phi(x)\;\dx +\int_{\Omc}\hspace{-2ex}\psi(y)\mathcal N \phi(y)\;\dy,
\end{align}
where,  $\mathcal{N}\phi $ denotes  the {\em nonlocal normal derivative $\mathcal N$} of $\phi$ across the boundary of  $\Omega$ with respect to $\nu$ and is defined by
\begin{align}\label{eq:nonlocal-normal-derivative}
\mathcal{N} \phi(x)\coloneqq \int_{\Omega}\left(\phi(x)-\phi(y)\right)\nu(x-y)\; \dy,~~~~x\in\RR^N\setminus\Om.
\end{align} 
\medskip

\noindent With the aforementioned function spaces at hand, we are now in position to define the notion of  weak solutions to  the problem \eqref{eq:main-equation}.
\begin{definition}
Let $\varphi$ satisfies Assumption~\ref{eq:assumption-phi} and $u_0\in L^2(\Omega)$. A function
$u:\Omega_T \to \mathbb{R}$ is said to be a weak solution of problem \eqref{eq:main-equation}, if
\begin{enumerate}[$(i)$]
\item
$u\in L^2(0,T;\mathbb{X}_{\nu}(\Omega \vert \RR^N))$ and  $\partial_t u\in L^2(0,T;\big( \mathbb{X}_{\nu}(\Omega \vert \RR^N)\big)^*)$; 
\item  for every $\psi\in C^1_c\big([0,T);\mathbb{X}_{\nu}(\Omega \vert \RR^N) \big)$ (i.e., $\psi(\cdot, T)=0$),  $u$ satisfies $u(\cdot, 0)=u_0$ and 
\begin{align}\label{eq:def-weak-solution} 
\int_\Omega  \partial_t u\,\psi\,\dx + \upxi(u,\psi) + \int_\Omega \varphi(v)\,u\,\psi\dx\, = 0\quad \text{for all $0\leq t\leq T$}.
\end{align}
In particular, we have 
\[
\int_0^T \int_\Omega   u\,\partial_t\psi\,\dx\, \dt + \int_0^T \upxi(u,\psi)\,\dt +\int_0^T \int_\Omega \varphi(v)\,u\,\psi\dx\, \dt = \int_\Omega u_0 \psi_0\d x.
\] 
\end{enumerate}
\end{definition}

\noindent By the density of $ C_c^\infty(\Omega)$ in $\mathbb{X}_{\nu}(\Omega \vert \RR^N)$, it is sufficient to take $\psi\in C_c^1([0,T); C_c^\infty(\Omega))$ as the test functions in \eqref{eq:def-weak-solution}.  Our proof of the existence of a weak solution to the problem \eqref{eq:main-equation} relies upon the following Tychonoff fixed-point Theorem \ref{thm:tychonoff} which is a generalization of the Brouwer and Schauder fixed-point theorems.

\begin{theorem}[Tychonoff \cite{Ty35}]\label{thm:tychonoff} Let $X$ be a reflexive separable Banach space. Let $G\subset X$ be closed convex and bounded set. Any weakly sequentially continuous map $\pi:G\to G$ has a fixed point. \linebreak[-2]
\end{theorem} 
\noindent We  emphasize that when $G$ is compact and convex, Theorem
\ref{thm:tychonoff} is known as the Schauder fixed-point theorem, while, when $X$ is of  finite dimension it is known as the Brouwer fixed-point theorem. 

\section{Nonlocal elliptic and parabolic problem}\label{Sec:auxiliaries}
The overreaching goal of this section is to investigate weak solutions to two specific nonlocal problems which is of interest in the proof of our main result. The first problem is an elliptic nonlocal problem and the second one is a parabolic nonlocal problem.

\subsection{Nonlocal elliptic problem}\label{sec:elliptic-problem}
Given a measurable function $f:\Omega\to \mathbb{R}$, we consider the elliptic problem consisting into finding a function $v:\R^N\to \mathbb{R}$ satisfying 
of the following problem:
\begin{align}\label{eq:elliptic-equation}
\begin{cases}
\mathscr{L} v + \varphi(v)\,v= f  &\mbox{ in }\; \Omega,\\
v =0 &\mbox{ on }\; \Omc. 
\end{cases}
\end{align}
Heuristically, the problem \eqref{eq:elliptic-equation} results from the evolution problem \eqref{eq:main-equation} by  integrating with respect to $t$ from $0$ to $T$. In a sense, the functions $v$ and $f$ correspond to $ \int_0^T u(\cdot,t)\,\dt$ and $u_0-u(\cdot,T)$, respectively.
Semilinear problems of type \eqref{eq:elliptic-equation} are considered in the classical scenario in \cite{BB78, Gos71,He73,St21}) with the operator $\mathscr L$ replaced with  $-\Delta$. There, the difficulties with the integrability of the term $\varphi(v)v$ were handled. In our case, we consider the function $f\in L^2(\Omega)$, so that we expect more from the solution of the problem such as $\varphi(v)\in L^2(\Omega)$. 
We need to introduce  the following notation $$\upchi(\tau)=\varphi(\tau)\tau.$$ A function $v\in \mathbb{X}_{\nu}(\Omega \vert \R^N)$ is said to be a weak solution of problem \eqref{eq:elliptic-equation} if $\upchi(v)\in L^2(\Omega)$ and 
\begin{align}\label{eq:weak-sol-elliptic}
\mathcal{E}(v,\psi)+(\upchi(v),\psi)=(f,\psi)\quad\text{for all}\quad\psi\in \mathbb{X}_{\nu}(\Omega \vert \R^N).
\end{align}
Next, we want to  show that the above variational problem \eqref{eq:weak-sol-elliptic} is well-posed in the sense of Hadamard. In other words, it possesses a unique solution which continuously depends upon the data. Let us start with the following stability lemma.

\begin{lemma}\label{lem:stability} Let $f_i\in L^2(\Omega)$, $i=1,2$. Assume that $v_i\in \mathbb{X}_{\nu}(\Omega \vert \R^N)$
	satisfies 
	\begin{align*}
	\mathcal{E}(v_i,\psi)+(\upchi(v_i),\psi)= (f_i,\psi)\quad\text{for all}\quad\psi\in \mathbb{X}_{\nu}(\Omega \vert \R^N).
	\end{align*}
	Then for some constant $C= C(N,\Omega, \nu)>0$ only depending only $N, \Omega$ and $\nu$ such that   
	\begin{align*}
	\|v_1-v_2\|_{\mathbb{X}_{\nu}(\Omega \vert \R^N)}\leq C 	\|f_1-f_2\|_{(\mathbb{X}_{\nu}(\Omega \vert \R^N))^*}.
	\end{align*}
\end{lemma} 
\begin{proof}
	Combining both equation and testing with $\psi= v_1-v_2$ yields 
	\begin{align*}
	\mathcal{E}(v_1-v_2,v_1-v_2)+\big(\upchi(v_1)-\upchi(v_2), v_1-v_2\big)_{L^2(\Omega)}=(f_1-f_2, v_1-v_2)_{L^2(\Omega)}. 
	\end{align*}
	Observing that $\tau\mapsto \upchi(\tau)=\varphi(\tau)\tau$ is non-decreasing, is equivalent to saying that 
	\begin{align}\label{eq:monotone}
	(\upchi(\tau_1)-\upchi(\tau_2))(\tau_1-\tau_2)\geq 0\qquad\text{for all $\tau_1,\tau_2\in \R, $}
	\end{align}
	the above relation implies 
	\begin{align*}
	\|v_1-v_2\|^2_{\mathbb{X}_{\nu}(\Omega \vert \R^N)}\leq \|f_1-f_2\|_{(\mathbb{X}_{\nu}(\Omega \vert \R^N))^*}\|v_1-v_2\|_{\mathbb{X}_{\nu}(\Omega \vert \R^N)}.
	\end{align*}
	The desired estimate follows from the Poincar\'e-Friedrichs inequality \eqref{eq:poincareX}. 
\end{proof}

\vspace{1mm}
	
\noindent The next result  reminisces \cite[Lemma 1, Section 3.1]{St21} in the nonlocal setting. 
\begin{theorem}\label{thm:existence-elliptic}
	Let Assumption~\ref{eq:assumption-phi} be in force and let $f\in L^2(\Omega)$. Then the  problem \eqref{eq:elliptic-equation} has a unique weak solution $v\in \mathbb{X}_{\nu}(\Omega \vert \R^N)$. Moreover, the following estimates hold true: 
	\begin{enumerate}[$(i)$]
		\item
		$\mathcal{E}(v,v)\leq C\,\|f\|^{2}_{L^{2}(\Omega)}$ where $C>0$  only  depends on $N$, $\Omega$, and $\nu$;
		\item  $\|\varphi(v)\,v\|_{L^{2}(\Omega)}\leq\|f\|_{L^{2}(\Omega)}$; 
		\item
		$\|\varphi(v)\|^2_{L^{2}(\Omega)}\leq  \frac{1}{\delta^2} \|f\|^2_{L^2(\Omega)} + |\Omega|$, with $\delta>0$ only depending on $\varphi$. 
	\end{enumerate}
\end{theorem}

\smallskip

\begin{proof}
	Note that the uniqueness immediately follows from Lemma \ref{lem:stability}.  We prove the remaining results of Theorem \ref{thm:existence-elliptic} in several steps. Our proof  follows that of \cite[Lemma 1, Section 3.1]{St21}.
	\vspace{1mm}
	
	\noindent {\bf Step 1:} We are interested in establishing the well-posedness of problem \eqref{eq:elliptic-equation} using the Galerkin method which consists into projecting the latter on suitable finite dimensional space.  First of all, we  mention that bounded functions are dense in $V_\nu(\Omega|\R^N)$ and hence in $\mathbb{X}_\nu(\Omega|\R^N)$. Thus, there is an orthonormal basis $\{\phi_k\}$  of $\mathbb{X}_\nu(\Omega|\R^N)$ whose elements are bounded, i.e., $\phi_k\in L^\infty(\Omega)$. 
	
	\noindent   We emphasize that the inner product in $\mathbb{X}_{\nu}(\Omega \vert \R^N)$  is defined as $(\psi_1,\psi_2)_{\mathbb{X}_{\nu}(\Omega \vert \R^N)}= \mathcal{E}(\psi_1,\psi_2)$ for $\psi_1$, $\psi_2\in \mathbb{X}_{\nu}(\Omega \vert \R^N)$. 
	Let $\mathcal{V}_{k}$ be the subspace of $\mathbb{X}_\nu(\Omega|\R^N)$ spanned by the basis functions $\{\phi_1,\ldots, \phi_{k}\}$. For each $k\in \mathbb{N}$, we claim the existence of a function  $v_k\in \mathcal{V}_{k}$ such that
	\begin{align}\label{eq:weak-finite-dim}
	\mathcal{E}(v_{k},\psi) +(\upchi(v_k),\psi)=(f,\psi)\quad\text{for all}\quad\psi\in \mathcal{V}_{k}.
	\end{align}
	We prove this in two different ways. First, note that \eqref{eq:weak-finite-dim} is equivalent to the minimization problem
	\begin{align*}
	\mathcal{J}(v_k)=  \min_{w\in \mathcal{V}_k}  \mathcal{J}(w)\quad \text{with}\quad  \mathcal{J}(w):=  \frac12 \mathcal{E}(w,w) + \int_\Omega G(w)\d x + \int_\Omega fw\d x
	\end{align*}
	where we  define the function $G(w)= \int_0^w\chi(\tau)\d \tau= \int_0^w\varphi(\tau)\tau \d \tau $.  Note that $G$ is non-negative since $\varphi(\tau)\geq 0$ and that the mapping $w\mapsto \mathcal{J}(w)$ is continuous on $\mathcal{V}_k$. Furthermore, with the aid of the Poincar\'e-Friedrichs inequality \eqref{eq:poincareX} we find that $\mathcal{J}(w)\to \infty$, as $\|w\|_{\mathbb{X}_\nu(\Omega|\R^N)} \to \infty$ and $w\in \mathcal{V}_k$. Since $\dim \mathcal{V}_k<\infty$, the existence of a minimizer $v_k\in\mathcal{V}_k$ of $\mathcal{J}$ springs from folklore arguments. 

	\vspace{1mm}
	\noindent Alternatively, as highlighted in \cite{St21},  we obtain the existence of $v_k$ using  the Brouwer fixed-point theorem as follows. Let $w\in \mathcal{V}_k$, necessarily $\varphi(w)$ is a bounded function since $\phi_k $'s are also bounded. The Lax-Milgram lemma implies there is a unique function $\widehat{w}\in  \mathcal{V}_k$
	such that 
	\begin{align*}
	\mathcal{E}(\widehat{w},\psi) +(\varphi(w)\widehat{w},\psi)=(f,\psi)\quad\text{for all}\quad\psi\in \mathcal{V}_{k}.
	\end{align*}
	
	\noindent In particular, the Poincar\'{e}--Friedrichs inequality \eqref{eq:poincareX} yields
	\begin{align*}
	\begin{aligned}
	\mathcal{E}(\widehat{w}, \widehat{w}) + \int_\Omega \varphi(w )\widehat{w}^2\,\dx\leq &\|f\|_{L^{2}(\Omega)}\,\|\widehat{w}\|_{L^{2}(\Omega)}\leq  C\|f\|_{L^{2}(\Omega)}\, \|\widehat{w}\|_{\mathbb{X}_{\nu}(\Omega \vert \R^N)} 
	\end{aligned}
	\end{align*}
	Thus, letting $R=C\,\|f\|_{L^{2}(\Omega)}$, since $\varphi\geq0 $ we obtain the following estimates
	\begin{align}\label{eq:boundedmapT}
	&\|\widehat{w}\|_{\mathbb{X}_{\nu}(\Omega \vert \R^N)}\leq R
	\quad\text{ and }\quad
	\int_\Omega \varphi(w )\widehat{w}^2\,\dx\leq R^2.
	\end{align}
	We let $ \mathcal{B}_R=\big\{ w\in \mathcal{V}_k: \|w\|_{\mathbb{X}_\nu(\Omega|\R^N)} \leq R\big\}$, be the closed ball in $\mathcal{V}_k$ of radius $R$ centered at the origin.  Clearly, \eqref{eq:boundedmapT} implies that the mapping $T:\mathcal{V}_k\to \mathcal{B}_R$ with $Tw=\widehat{w}$  is well defined. 
	It remains to prove that $T$ is a continuous mapping. Indeed, let $\{w_n\}$ be a sequence in $\mathcal{V}_k$ with $w_n= \lambda_{1,n}\phi_1+\cdots+  \lambda_{k,n}\phi_k$ converging in $\mathcal{V}_k$ to a function $w= \lambda_1\phi_1+\cdots+  \lambda_k\phi_k$, i.e., $\lambda_{\ell,n}\xrightarrow{n\to\infty } \lambda_\ell$, $\ell=1,2,\cdots,k$. 
	By continuity we have $ \varphi(w_n)\xrightarrow{n\to \infty}\varphi(w)$ almost everywhere. In addition, the convergence in $L^2(\Omega)$ also holds, i.e., $ \|\varphi(w_n)-\varphi(w)\|_{L^2(\Omega)}\xrightarrow{n\to \infty}0$ since the continuity gives $\sup_{n\geq 0} \|\varphi(w_n) \|_{L^\infty(\Omega)} <\infty$ because $\sup_{n\geq 0}\|w_n\|_{L^\infty(\Omega) } <\infty$. On the other side, in virtue of  the first estimate in \eqref{eq:boundedmapT}, the sequence 
	$\{Tw_n\}$ is bounded in finite dimensional space $\mathcal{V}_k$ and thus converges in $\mathcal{V}_k$ up to a subsequence  to some $w_*\in \mathcal{V}_k$. Altogether, it follows that, for all  $\psi\in \mathcal{V}_k\subset L^\infty(\Omega)$ 
	\begin{align*}
	(f,\psi)= \lim_{n\to \infty} \mathcal{E}(\widehat{w}_n,\psi) 
	+(\varphi(w_n)\widehat{w}_n,\psi)=
	\mathcal{E}(w_*,\psi) +(\varphi(w)w_*, \psi). 
	\end{align*}
	The uniqueness of $\widehat{w}$ entails that  $w_*=\widehat{w}=Tw$ and hence  the whole sequence $\{Tw_n\}$ converges in $Tw$ in $\mathcal V_k$, which gives the continuity of $T$. 
	Therefore, by the Brouwer fixed-point theorem,  $T$ has a  fixed point $v_k\in \mathcal V_k$, i.e., $v_k=Tv_k$ which clearly  satisfies \eqref{eq:weak-finite-dim} as announced.   
	
	\noindent To continue, we must show that a subsequence of $\{v_k\}$ converges in $L^2(\Omega)$. To do this, we recall that $R= C\,\|f\|_{L^{2}(\Omega)}$ so that  from \eqref{eq:boundedmapT} we get the following estimates
	\begin{align} \label{eq:bound-chi}
	\|v_k\|_{\mathbb{X}_{\nu}(\Omega \vert \R^N)}\leq R
	\quad \text{and}\quad
	\int_\Omega \upchi(v_k)\, v_k\, \dx \leq R^2 \qquad\text{ for all $k\in \mathbb{N}$}.
	\end{align}
	Therefore, the sequence $\{v_k\}$ is clearly bounded in $\mathbb{X}_{\nu}(\Omega \vert \R^N)$. The compactness Theorem \ref{thm:compactness} yields the existence of a subsequence, still denoted by $\{v_k\}$,  converging weakly in $\mathbb{X}_{\nu}(\Omega \vert \R^N)$, strongly  in $L^{2}(\Omega)$ and almost everywhere in $\Omega$ to a function $v$. Wherefore, due to the continuity of $\upchi$, we get
	\begin{align}\label{eq:chi-conv-a.e}
	\upchi(v_k)\to\upchi(v)\quad \text{almost everywhere in} ~~\Omega.
	\end{align}
	
	\noindent {\bf Step 2:} Next, we prove that the functions $\{\upchi(v_k)\}$ are uniformly integrable. 
	In view of the estimate \eqref{eq:bound-chi},  for each measurable set $\Gamma\subset\Omega$ and each $\Lambda>0$, we  let
	$\Gamma_\Lambda^k=\{x\in \Gamma\,:\, |v_k(x)|\ge \Lambda\}$ so that
	\[
	\int_{\Gamma_\Lambda^k}|\upchi(v_k)|\,\dx \le \frac{1}{\Lambda}\int_{\Omega}\upchi(v_k)\, v_k\, \dx\le \frac{R^2}{\Lambda}.
	\]
	
	\noindent Since $\upchi$ is  non-decreasing, putting  $\upgamma(\Lambda)=\Lambda\max \{\varphi(-\Lambda),\varphi(\Lambda)\}$, we get  
	\begin{align*}
	|\upchi(\tau)|\le \upgamma(\Lambda) \quad \text{for all} \quad \tau\in[-\Lambda,\Lambda]. 
	\end{align*}
	Therefore, the following relation holds
	\[
	\int_{\Gamma\setminus \Gamma_\Lambda^k} |\upchi(v_k)|\,\dx \le \upgamma(\Lambda)\,|\Gamma|,
	\]
	where $|\Gamma|$ is the Lebesgue measure of the set $\Gamma$. These inequalities imply that
	\[
	\int_\Gamma|\upchi(v_k)|\,\dx\le \frac{R^2}{\Lambda} + \upgamma(\Lambda)\,|\Gamma|.
	\]
Thus, for an arbitrary $\varepsilon >0$, we take $\Lambda=2R^2/\varepsilon$ and $\updelta=\varepsilon/(2\upgamma(\Lambda))$.
	Therefore, we find that
	$$\sup_{k\geq1}\int_{\Gamma}|\upchi(v_k)|\,\dx< \varepsilon$$ for an arbitrary measurable set $\Gamma\subset\Omega$ such that $|\Gamma|<\delta$. This, is precisely the uniform integrability of $\upchi(v_k)$.
	This fact together with \eqref{eq:chi-conv-a.e} and the Vitali convergence theorem (see, e.g., \cite[Theorem A.19]{Foghem})
	enable us to conclude that $\upchi(v) \in L^1(\Omega)$ and
	$\upchi(v_k)\to \upchi(v)$ in $L^1(\Omega)$ as $k\to\infty$.
	Now passing to the limit in \eqref{eq:weak-finite-dim} as $k\to\infty$ we find that $v$ satisfies \eqref{eq:weak-sol-elliptic}, which along  with Lemma \ref{lem:stability}, means that $v$ is a unique weak solution of problem~\eqref{eq:elliptic-equation}.\medskip
	
	\noindent {\bf Step 3:} We prove the estimates in $(i), (ii)$ and  $(iii)$. 
The estimate (i) follows from the first inequality in \eqref{eq:bound-chi} since the weak convergence of $(v_k)_k$ implies 
$$\mathcal{E}(v,v)\leq \liminf\limits_{k\to \infty}\mathcal{E}(v_k,v_k)\leq C\|f\|^2_{L^2(\Omega)}.$$
Next, let us consider the truncation  $\bar{v}_\ell= \max(-\ell, \min(\ell ,v)),$  $\ell\geq 1$. Then $\bar{v}_\ell\in \mathbb{X}_{\nu}(\Omega \vert \R^N)$ since $(v(x)-v(y)) (\bar{v}_\ell(x)-\bar{v}_\ell(y)\geq |\bar{v}_\ell(x)-\bar{v}_\ell(y)|^2$. The latter inequality and \eqref{eq:monotone} imply $\mathcal{E}(v,\upchi(\bar{v}_\ell)) \geq \mathcal{E}(\bar{v}_{\ell},\upchi(\bar{v}_\ell)) \geq 0$.  Moreover, $|\upchi(\bar{v}_\ell)(x)-\upchi(\bar{v}_\ell)(y)|\leq c |\bar{v}_\ell(x)-\bar{v}_\ell(y)|$ since $\upchi$ is Lipschitz on $[-\ell,\ell]$ and $|\bar{v}_\ell|\leq\ell$.  Thus, $\upchi(\bar{v}_\ell)\in \mathbb{X}_{\nu}(\Omega \vert \R^N)$.
	Noticing that, $\upchi(v)\upchi(\bar{v}_\ell)\ge \upchi^2(\bar{v}_\ell),$  taking $\psi=\upchi(\bar{v}_\ell)$ in \eqref{eq:weak-sol-elliptic} yields
	\begin{align*}
	\|\upchi(\bar{v}_\ell)\|^2_{L^2(\Omega)} 
	&\leq \mathcal{E}(v,\upchi(\bar{v}_\ell)) +(\upchi(v), \upchi(\bar{v}_\ell)) 
	\le \|f\|_{L^2(\Omega)} \|\upchi(\bar{v}_\ell)\|_{L^2(\Omega)}.
	\end{align*}

\noindent Thus, since $\{\upchi^2(\bar{v}_\ell)\}\to\upchi^2(v)$ a.e. in $\Omega$, Fatou's lemma implies  the second estimate as follows  
\begin{align*}
	\|\upchi(v)\|_{L^2(\Omega)} \leq \liminf_{\ell\to \infty}\|\upchi(\bar{v}_\ell)\|_{L^2(\Omega)} \leq \|f\|_{L^2(\Omega)}.
	\end{align*}
	
	\noindent Finally, by continuity of  $\varphi$, there exists $\delta>0$ such that $\varphi^2(\tau)\leq 1$ for all $\tau\in[-\delta,\delta]$. Hence,  letting $\Gamma_{\delta}=\{x\in\Omega\;:\; |v(x)|\ge \delta\}$,  the second estimate implies the third one as follows 
	\begin{align*}
	\int_{\Omega} \varphi^2(v)\,\dx&= \int_{\Gamma_\delta} \varphi^2(v)\,\dx +\int_{\Omega\setminus \Gamma_\delta} \hspace{-2ex}\varphi^2(v)\,\dx 
	\leq \frac{1}{\delta^2}\int_\Omega \varphi^2(v)\,v^2\,\dx+\int_{\Omega\setminus \Gamma_\delta}\hspace{-2ex} \varphi^2(v)\,\dx  \leq \frac{1}{\delta^2}\|f\|^2_{L^2(\Omega)} +|\Omega|.
	\end{align*}
\end{proof}

\vspace{1mm}

\noindent Next, we define the mapping $V:L^2(\Omega)\to \mathbb{X}_{\nu}(\Omega \vert \R^N)$ such that, for $f\in L^2(\Omega)$, 
\begin{align}\label{eq:def-v}
\text{$v=V(f)$ is the unique weak solution of problem \eqref{eq:elliptic-equation}.}
\end{align} 
We derive in the lemma below, some convergence results for the sequence $\{V(f_k)\}$ which are decisive for the  application the Tychonoff fixed-point Theorem \ref{thm:tychonoff}.
\begin{lemma}\label{lem:weak-convergence}
	Assume  that $f_k\rightharpoonup f$ weakly $L^2(\Omega)$ and  let $v_k=V(f_k)$ and $v=V(f)$. Then  it holds that  $v_k\to v$ strongly in $\mathbb{X}_{\nu}(\Omega \vert \RR^N)$ and  $\varphi(v_k)\rightharpoonup \varphi(v)$ weakly in $L^2(\Omega)$.
\end{lemma}

\begin{proof}
	Let us identify $f_k-f$ in $  (\mathbb{X}_{\nu}(\Omega \vert \R^N))^*$ with the linear  form
	\begin{align*}
	w\mapsto \int_\Omega (f_k(x)-f(x))w(x)\d x.
	\end{align*}
By the reflexivity of $\mathbb{X}_{\nu}(\Omega \vert \RR^N)$,
 there exists $w_k\in X_{\nu}(\Omega \vert \RR^N)$  (cf. \cite[Theorem 2]{Ja64}, \cite[Chapter 6]{KK11} or \cite[page 60]{Che13}) such that  
$\|w_k\|_{\mathbb{X}_{\nu}(\Omega \vert \RR^N)} \leq 1$ and 
	\begin{align*}
	\|f_k-f\|_{\mathbb{X}_{\nu}((\Omega \vert \R^N))^*} = \int_\Omega (f_k(x)-f(x)) w_k(x)\d x.
	\end{align*}
	According to the compactness Theorem \ref{thm:compactness} we may assume that $\{w_k\}$ strongly converges to  some $w$ in $L^2(\Omega)$. Therefore, the weakly convergence of $\{f_k\}$ implies that
	\begin{align*}
	\|f_k-f\|_{\mathbb{X}_{\nu}((\Omega \vert \R^N))^*} = \int_\Omega (f_k(x)-f(x)) w_k(x)\d x\xrightarrow{k\to \infty}0. 
	\end{align*}
	The convergence in $\mathbb{X}_{\nu}(\Omega \vert \R^N)$ follows immediately from Lemma \ref{lem:stability} since
	\begin{align*}
	\|v_k-v\|_{\mathbb{X}_{\nu}(\Omega \vert \R^N)}\leq C 	\|f_k-f\|_{(\mathbb{X}_{\nu}(\Omega \vert \R^N))^*}\xrightarrow{k\to \infty}0. 
	\end{align*}
	On the other hand, we also have the strong  convergence of $\{v_k\}$
	in $L^2(\Omega)$ and the continuity of $\varphi$ imply  that $\{\varphi(v_k)\}$ converges almost everywhere to $ \varphi(v)$ up a subsequence. Furthermore, since $\{f_k\}$ is bounded, as in the proof of Lemma \ref{eq:elliptic-equation}, one easily gets that 
	\begin{align*}
	\|\varphi(v_k)\|_{L^2(\Omega)}\leq C\quad\text{for all $k\geq 1$,}
	\end{align*}
	for a constant $C>0 $ independent on $k$. Thus, $\{\varphi(v_k)\}$ has a further subsequence weakly converging in $L^2(\Omega)$. The Banach-Saks Theorem, see \cite[Appendix A]{MR12} or \cite[Proposition 10.8]{Ponce16},  infers the  existence of a further subsequence whose  C\'esaro mean converges strongly in $L^2(\Omega)$ and almost everywhere  in $\Omega$ to the same limit.  Necessarily,  since $\{\varphi(v_k)\}$ converges almost everywhere to $ \varphi(v)$, the entire sequence $\{\varphi(v_k)\}$ weakly converges in $L^2(\Omega)$ to $\varphi(v)$. 
\end{proof}

\subsection{Nonlocal parabolic problem}\label{sec:parabolic}
We consider the following parabolic problem:
\begin{align}\label{eq:weight-parabolic}
\begin{cases}
\partial_t u +\mathscr{L} u + \zeta u=0 &\mbox{ in }\; \Omega_T,\\
u = 0 &\mbox{ in }\; \Sigma , \\
u(\cdot,0) = u_{0}, &\mbox{ in }\; \Omega,
\end{cases}
\end{align}
where $u_0 , \zeta\in L^2(\Omega)$ with  $\zeta \geq 0$. We also assume $\nu\not\in L^1(\R^N)$ so that by Theorem \ref{thm:compactness},  the embedding $\mathbb{X}_{\nu}(\Omega \vert \R^N) \hookrightarrow L^2(\Omega)$ is compact.   
Therefore, by the standard Galerkin superposition  method (see for instance \cite[Section 4.6]{Foghem}), a weak solution $u$ of the problem \eqref{eq:weight-parabolic} can be easily obtained in $L^2\big(0,T;\mathbb{X}_{\nu}(\Omega \vert \R^N) \cap L^2(\Omega,\zeta)\big)$. Here $L^2(\Omega,\zeta)$ is the Hilbert space with the norm 
\begin{align*}
\|u\|_{L^2(\Omega, \zeta)}^2=\int_\Omega |u(x)|^2\,\zeta(x)\dx.
\end{align*}
We omit the proof as well as various justifications (see also \cite{Ev10,GGZ74}).  
Another possibility, is to observe that (see \cite{Ka13}) there exists a unique semigroup  with generator $A$ on $L^2(\Omega)$ associated to the closed bilinear form $a(u,v)= (u, v)_{\mathbb{X}_{\nu}(\Omega \vert \R^N)} + (u,v)_{L^2(\Omega, \zeta)}$, with $u, v\in \mathbb{X}_{\nu}(\Omega \vert \R^N) \cap L^2(\Omega, \zeta) $,  such that $ a(u,v) = \langle Au, v\rangle$. Thus $u(x,t)= e^{-tA}u_0(x), \, \, 0\leq t\leq T$, is the unique weak solution to \eqref{eq:weight-parabolic}. The weak solution of problem \eqref{eq:weight-parabolic} satisfies the energy estimate:
\begin{align}\label{eq:weight-para-esti}
\frac{1}{2}\, \|u(\cdot, t)\|^2_{L^2(\Omega)} +\int_0^t \mathcal{E}(u,u)\, \mathrm{d}\tau + \int_0^t \int_\Omega \zeta\, u^2\,\dx\, \mathrm{d}\tau
\le\frac{1}{2}\, \|u_0\|^2_{ L^2(\Omega)}
\end{align}
for all $t\in[0,T)$. Besides that, $\partial_t u$ belongs to the space $L^2\big(0,T;(\mathbb{X}_{\nu}(\Omega \vert \R^N)\cap L^2(\Omega,\zeta))^*\big)$,
where $(\mathbb{X}_{\nu}(\Omega \vert \R^N)\cap L^2(\Omega,\zeta))^*$ is the dual space of $\mathbb{X}_{\nu}(\Omega \vert \R^N)\cap L^2(\Omega,\zeta)$.
According to Proposition \ref{prop:emb} (see \eqref{contWTA}) we find that $u\in C(0,T;L^2(\Omega))$.
Thus, the function $u_T=u(\cdot, T)$ is well defined and  belongs to $L^2(\Omega)$. Moreover,  the estimate \eqref{eq:weight-para-esti} holds in particular for  $t=T$. 

\vspace{2mm}

\noindent For each $\zeta\in L^2(\Omega)$, $\zeta\geq 0$,  define  the mapping $\mathscr{U}: \zeta \mapsto \mathscr{U}(\zeta)$ where  
$\mathscr{U}(\zeta)\in L^2\big(0,T;\mathbb{X}_{\nu}(\Omega \vert \R^N) \cap L^2(\Omega,\zeta)\big)$
is the unique weak solution of problem \eqref{eq:weight-parabolic}.\medskip

\noindent We want to  study the continuity of the operators $\mathscr{U}$ and $\mathscr{U}_T$ on $L^2(\Omega)$, with $\mathscr{U}_T(\zeta) (\cdot)= \mathscr{U}(\zeta)(\cdot, T)$.  
\begin{lemma}\label{lem:weak-conv-bis}
	Let $u_0\in L^2(\Omega)$ and $\{\zeta_k\}$ be a sequence of non-negative functions converging weakly in $L^2(\Omega)$ to a function $\zeta$. Then $\mathscr{U}_T(\zeta_k)\rightharpoonup \mathscr{U}_T(\zeta)$ weakly in $L^2(\Omega)$ as $k\to\infty$.
\end{lemma}
\begin{proof}
	For brevity, we denote  $u_k=\mathscr{U}(\zeta_k)$ and $u=\mathscr{U}(\zeta)$. It follows from \eqref{eq:weight-para-esti}, for all $k\in \mathbb{N}$,
	\begin{align}\label{eq:bounded-uk}
	\frac{1}{2}\, \|u_k(\cdot, T)\|^2_{L^2(\Omega)} +\int_0^T \mathcal{E}(u_k,u_k)\, \mathrm{d}t + \int_0^T \int_\Omega \zeta_k\, u_k^2\,\dx\, \mathrm{d}t\leq \frac{1}{2}\, \|u_0\|^2_{L^2(\Omega)}.
	\end{align}
\noindent Next, consider $\psi:\R^N\times (0,T)\to \mathbb{R}$ be an arbitrary smooth function such that $\psi = 0$ in $\R^N\setminus \Omega \times (0,T)$. Define $w_k(x)=\int_0^T u_k(x,t) \psi(x,t)\dt$ assuming $\psi(\cdot, t)\in C_c^\infty(\Omega)$ there holds  $$|\psi(x,t)-\psi(y,t)|\leq \|\psi(\cdot, t)\|_{W^{1,\infty}(\R^N)}(1\land |x-y|).$$
Using this inequality yields 
\begin{align*}
|w_k(x)-w_k(y)|^2&=\Big| \int_0^T( u_k(x,t) -u_k(y,t)) \psi(x,t)\ \d t + \int_0^T (\psi(x,t)-\psi(y,t)) u_k(y,t) \d t \Big|^2\\
&\leq 2T \int_0^T| u_k(x,t) -u_k(y,t)|^2|\psi(x,t)|^2\d t+ 2T \int_0^T| \psi(x,t)-\psi(x,t)|^2|u_k(y,t)|^2\d t\\
&\leq 2T\sup_{t\in [0,T]}\|\psi(\cdot, t)\|^2_{W^{1,\infty}(\R^N)}\int_0^T| u_k(x,t) -u_k(y,t)|^2 + (1\land|x-y|^2)|u_k(y,t)|^2\d t. 
\end{align*}
Integrating both side over $\Omega\times\R^N$ (see Remark \ref{rmk:equivalence}) with respect to the measure $\nu(x-y)\d x\d y$ gives 
	\begin{align*}
	\mathcal{E} (w_k, w_k) 	&\leq 2T\sup_{t\in [0,T]}\|\psi(\cdot, t)\|^2_{W^{1,\infty}(\R^N)} \int_0^T\mathcal{E} (u_k, u_k)\d t +\int_{\R^N} (1\land|h|^2)\nu(h)\d h\int_0^T\int_\Omega |u_k(,t)|^2\d x\d t.
	\end{align*}
Altogether, with the Poincar\'e inequality \eqref{eq:poincareX} and the inequality \eqref{eq:bounded-uk} yield 
	\begin{align*}
	\frac{1}{2}\, \|u_k(\cdot, T)\|^2_{L^2(\Omega)} + \int_0^T \mathcal{E}(u_k,u_k)\, \mathrm{d}t &\leq \frac{1}{2}\, \|u_0\|^2_{L^2(\Omega)}\quad\text{and}\quad
\mathcal{E} (w_k, w_k) 	\leq C_\psi\|u_0\|^2_{L^2(\Omega)}
\end{align*}
\noindent Therefore, $\{w_k\}$ is bounded in $\mathbb{X}_\nu(\Omega|\R^N)$. By the compactness Theorem \ref{thm:compactness} we can  assume that $\{w_k\}$ strongly converge to some $w\in L^2(\Omega)$. On the other hand, $\{u_k\}$ and $\{u_k(\cdot,T)\}$ are   bounded  in  $L^2\big(0,T;\mathbb{X}_{\nu}(\Omega \vert \R^N))$ and $L^2(\Omega)$ respectively and hence can be assumed to  weakly converge to some $u\in L^2\big(0,T;\mathbb{X}_{\nu}(\Omega \vert \R^N))$ and $h\in L^2(\Omega)$ respectively. In particular, the strong convergence implies that,  for every  $\phi\in C_c^\infty(\Omega)$, 
\begin{align*}
	\int_\Omega w(x)\phi(x)\d x=\lim_{k\to \infty}\int_\Omega w_k(x)\phi(x)\d x= \lim_{k\to \infty}\int_\Omega\phi(x) \int_0^Tu_k(x,t)\d t\d x= \int_\Omega\phi(x) \int_0^Tu(x,t)\d t\d x. 
\end{align*}
It turns out that $w= \int_0^Tu(\cdot,t)\d t$ a.e. on $\Omega$. By the same token,  one gets $h= u(\cdot,T)$ a.e. on $\Omega$.
\noindent Once again, the strong convergence of $\{w_k\}$ and the weak of convergence of $\{\zeta_k\}$  in $ L^2(\Omega)$ yield 
	\begin{align}\label{eq:weak-con-pertubed}
	\int_0^T\int_\Omega \zeta_{k} u_{k} \psi\,\dx\,\dt \xrightarrow{k\to\infty} \int_0^T\int_\Omega \zeta u \psi\,\dx\, \dt.
	\end{align}
	
	\noindent For each $k\geq 1$, by definition of $u_k= \mathscr{U}(\zeta_k)$,  we get 
	\begin{align*}
	\int_0^T\int_\Omega  u_k \partial_t \psi\,\dx\,\dt + \int_0^T\mathcal{E}(u_k,\psi)\, \dt
	+\int_0^T\int_\Omega \zeta_k u_k\,\psi\,\dx\, \dt\,\dt =
	-\int_\Omega  (u_k(\cdot, T)\psi(\cdot, T)-u_0 \psi(\cdot, 0)) \,\dx.
	\end{align*}
	
	\noindent By choosing in  particular the test function $\psi(\cdot, T)=0$,  letting $k\to \infty$ yields $w(\cdot, 0) = u_0$ and 
	\begin{align*}
	&\int_0^T\int_\Omega \partial_t w \psi\,\dx+  \int_0^T\mathcal{E}(w,\psi)\, \dt+ \int_0^T\int_\Omega \zeta w \psi \d x\d t =0\quad \text{for all $\psi\in C^1_c(0,T; C_c^\infty(\Omega))$}. 
	\end{align*}
	This means that $u$ is a weak solution to  \eqref{eq:weight-parabolic} and by uniqueness, $u= \mathscr{U}(\zeta)$. The uniqueness of $u$ implies  that the whole sequence  $\{u_k\}$ weakly converges to $u$ in $ L^2\big(0,T;\mathbb{X}_{\nu}(\Omega \vert \R^N))$. Therefore, since $u_k=\mathscr{U}(\zeta_k)$ and $u=\mathscr{U}(\zeta)$ are weak solutions, using \eqref{eq:weak-con-pertubed} and  the weak convergence, for $\psi\in  L^2(\Omega) $ gives
	\begin{align*}
		\int_\Omega (\mathscr{U}_T(\zeta_k)-\mathscr{U}_T(\zeta))(x)\psi(x)\d x= \int_0^T\mathcal{E}(u_k-u,\psi)\, \dt+  \int_0^T\int_\Omega (u_k\zeta_k-u\zeta) (x,t)\psi(x)\d x\d t\xrightarrow{k\to\infty} 0.
	\end{align*}
	  That is,  the whole sequence $\{\mathscr{U}_T(\zeta_k)\}$ weakly converges to $U_T$ in $L^2(\Omega)$ which is the sought result.
\end{proof}

\section{Weak solvability and uniqueness of the solution}\label{Sec:WSU}
Armed with the above auxiliaries results, let us turn our attention to the proof of the weak solvability of problem \eqref{eq:main-equation}.  In order to apply the Tychonoff fixed-point Theorem \ref{thm:tychonoff}, we take $X=L^2(\Omega)$, $G=\{w\in L^2(\Omega)\,: \; \|w\|_{L^2(\Omega)}\le \|u_0\|_{L^2(\Omega)}\}$ which is clearly closed, convex and bounded.  The next result provides the existence of a weak solution to the problem \eqref{eq:main-equation}.  

\begin{theorem}\label{thm:existence}
	Let $u_0\in L^2(\Omega)$ and $T>0$.  Let the mapping $\pi:G\to G$ with $\pi (w)= \mathscr{U}_T(\varphi(v))$, where $v=V(u_0-w)$ (defined as in \eqref{eq:def-v}) is the unique weak solution to \eqref{eq:elliptic-equation} with $f= u_0-w$. 
	Then $\pi$ has a fixed point $u_T$ that is 
	\begin{align*}
	&u_T =\pi(u_T)= \mathscr{U}_{T} \big(\varphi(v)\big)\quad \text{with $v=V(u_0-u_T)$.}
	\intertext{Moreover, $\dis v=\int_0^T u \d t$ and }
	&u= \mathscr{U} \big(\varphi(v)\big)\quad\text{is a weak solution of the problem \eqref{eq:main-equation}. }
	\end{align*} 
\end{theorem}
\begin{proof}
For  $v=V(u_0-w) $ with $w\in G$, we know from Theorem \ref{thm:existence-elliptic} that $\varphi(v) \in L^2(\Omega)$. Thus for $\zeta = \varphi (v)\geq0  $,  the function $\mathscr{U}(\zeta)$ satisfies \eqref{eq:weight-para-esti} which implies that $\|\mathscr{U}_T(\zeta)\|_{L^2(\Omega)}\le \|u_0\|_{L^2(\Omega)}$. In particular,  $\|\pi(w)\|_{L^2(\Omega)}\le \|u_0\|_{L^2(\Omega)}$  for all $w\in G$ and thus, $\pi(G)\subset G$. It remains to prove the weak sequential continuity of $\pi$. Let $\{w_k\}$ be an arbitrary sequence in $G$ that converges to $w\in G$ weakly in $L^2(\Omega)$. We need to prove that $\pi(w_k)\rightharpoonup \pi(w)$ weakly in $L^2(\Omega)$ as $k\to\infty$.  In virtue of  Lemma~\ref{lem:weak-convergence}, $v_k=V(u_0-w_k)\to v=V(u_0-w)$  strongly in $\mathbb{X}_\nu(\Omega|\R^N)$ and $\varphi(v_k)\rightharpoonup\varphi(v)$ weakly in $L^2(\Omega)$ as
	$k\to\infty$, where $v_k=V(u_0-w_k)$ and $v=V(u_0-w)$.  In turn, Lemma~\ref{lem:weak-conv-bis} implies that $\pi (w_k)= \mathscr{U}_T(\varphi(v_k))\rightharpoonup\mathscr{U}_T(\varphi(v))=\pi (w)$ weakly in $L^2(\Omega)$ as $k\to\infty$. Thus, according to Theorem \ref{thm:tychonoff}, $\pi$ has a fixed point $u_T= \pi(u_T)$.   Next, knowing that $u_T$ is a fixed point of the mapping $\pi$, we show that $u=\mathscr{U}(\varphi(v))$, with $v=V(u_0-u_T) $, is a weak solution to the problem \eqref{eq:main-equation}. Indeed, recall that $u=\mathscr{U}(\varphi(v))$ is the unique weak solution to the problem \eqref{eq:weight-parabolic} with $\zeta= \varphi(v)$, \ie  $u(\cdot, 0) =u_0$ and for all $\psi\in C^1_c(0,T; \mathbb{X}_\nu (\Omega|\R^N) )$, 
	\begin{align}\label{eq:weak-main-bis}
	\int_0^T \int_\Omega  \partial_t u\,\psi\,\dx\, \dt + \int_0^T \mathcal{E}(u,\psi)\,\dt +\int_0^T \int_\Omega \varphi(v)\,u\,\psi\dx\, \dt = 0.
	\end{align}
Integrating by parts,  for $\psi \in \mathbb{X}_\nu (\Omega|\R^N)$ (time independent) we get 
\begin{align*}
	\mathcal{E}\Big(\int_0^Tu\d t,\psi\Big)+ \int_\Omega \varphi(v)\psi \,\int_0^Tu\, \d t\,\dx = \int_\Omega [u_0-u_T] \psi\,\dx. 
\end{align*}
	Thus, according to Theorem \ref{thm:existence-elliptic}, $v= \int_0^Tu(x,t)\dt$ is the unique  weak solution to the elliptic problem 
	\begin{align}\label{eq:elliptic-bis}
	\mathscr{L} v + \varphi(v)\,v= u_0-u_T  \,\, \mbox{ in }\; \, \Omega\quad 
	\text{and}\quad 	v =0\,\, \mbox{ on }\;\,  \Omc. 
	\end{align}
	We have shown that 
	$ v= V(u_0-u_T)=\int_0^Tu\dt. $
	Therefore,  we obtain
	\begin{align}\label{eq:existence-weak}
	u=\mathscr{U} \Big(\varphi\Big(\int_0^T u\, \dt\Big)\Big)
	\end{align}
	which, according to the relation \eqref{eq:weak-main-bis}, implies that $u$ is a weak solution to the problem \eqref{eq:main-equation}. 
\end{proof}

\noindent The main result of the paper is the following theorem.

\begin{theorem}\label{thm:main-result}
	Let  $u_0\in L^2(\Omega)$, $T>0$  and $\varphi$ satisfies Assumption~\ref{eq:assumption-phi}. The problem \eqref{eq:main-equation} has a weak solution $u\in
	L^\infty(0,T;L^2(\Omega))\cap L^2(0,T;\mathbb{X}_{\nu}(\Omega \vert \R^N))$ such that
\begin{align*}
	\varphi(v)\in L^2(\Omega),\quad \varphi(v)\, v\in L^2(\Omega), \quad \varphi(v)\, u^2 \in L^1(\Omega_T),~\text{and}
	\quad u\in C(0,T;L^2(\Omega)),
	\end{align*}
	where $\displaystyle v=\int_0^T u\,\dt$. Moreover, the following estimates hold true  
	\begin{align*}
	&\frac{1}{2}\, \|u\|^2_{L^\infty(0,T; L^2(\Omega))} +\|u\|^2_{L^2(0,T;\mathbb{X}_\nu(\Omega|\R^N))} + \int_0^T \int_\Omega \varphi(v)\, u^2\,\dx\, \d t
	\le\frac{1}{2}\, \|u_0\|^2_{ L^2(\Omega)},\\
	&\|\partial_t u\|^2_{L^2(0,T; (\mathbb{X}_\nu(\Omega|\R^N) \cap L^2(\Omega, \varphi(v))^*)} 
	\le\frac{1}{2}\, \|u_0\|^2_{ L^2(\Omega)}. 
	\end{align*} 
	
\end{theorem}

\begin{proof}
	The existence of a weak solution to the problem \eqref{eq:main-equation} follows immediately from Theorem \ref{thm:existence}. Furthermore, mimicking the  estimate \eqref{eq:weight-para-esti} yields 
	\begin{align*}
	\frac{1}{2}\, \|u\|^2_{L^\infty(0,T; L^2(\Omega))} +\|u\|^2_{L^2(0,T;\mathbb{X}_\nu(\Omega|\R^N))} + \int_0^T \int_\Omega \varphi(v)\, u^2\,\dx\, \d t
	\le\frac{1}{2}\, \|u_0\|^2_{ L^2(\Omega)}. 
	\end{align*}
	\noindent Now, for each $\psi\in L^2(0,T; \mathbb{X}_\nu(\Omega|\R^N) \cap L^2(\Omega, \varphi(v))$, by definition of $u$, we have 
	\begin{align*}
	&\left| \int_0^T\langle\partial_t u , \psi\rangle \d t\right|^{2}=\left| -\int_0^T \mathcal{E}(u,\psi)\, \d t - \int_0^T \int_\Omega \varphi(v)\, u\psi \,\dx\, \mathrm{d}t \right|^2
	\\
	&\leq \left(\int_0^T \mathcal{E}(u,u) \d t\right) \left(\int_0^T \mathcal{E}(\psi,\psi) \d t\right) + \left( \int_0^T \int_\Omega \varphi(v)\, u^2\dx \d t\right) \left( \int_0^T \int_\Omega \varphi(v) \psi^2\,\dx \d t\right)\\
	&\leq\frac{1}{2}\|u_0\|^2_{L^2(\Omega)} \left( \int_0^T \mathcal{E}(\psi,\psi)\, \d t+ \int_0^T \int_\Omega \varphi(v) \psi^2\,\dx \d t\right).
	\end{align*}
	This implies that 
	\begin{align*}
	\|\partial_t u\|^2_{L^2(0,T; (\mathbb{X}_\nu(\Omega|\R^N) \cap L^2(\Omega, \varphi(v))^*)} 
	\le\frac{1}{2}\, \|u_0\|^2_{ L^2(\Omega)}. 
	\end{align*}
	Therefore, we have $ u\in  L^2\big(0,T;\mathbb{X}_{\nu}(\Omega \vert \R^N) \cap L^2(\Omega,\zeta)\big)$ and $\partial_t u \in L^2\big(0,T;(\mathbb{X}_{\nu}(\Omega \vert \R^N) \cap L^2(\Omega,\zeta))^*\big)$ with $\zeta=\varphi\Big(\int_0^T u\d t\Big)$ which implies that $\varphi(v)u^2\in L^1(\Omega_T)$. On the one hand, by definition of $u= \mathscr{U}(\varphi(v))$ it follows that  $u\in C(0,T;L^2(\Omega) )$. On the other hand, we know that $ v=\int_0^T u\, \dt$ is the unique weak solution to the problem \eqref{eq:elliptic-bis} and hence from Theorem \ref{thm:existence-elliptic} we have $\varphi(v), \varphi(v)v\in L^2(\Omega)$.  This ends the proof. 
\end{proof}

\noindent Next, we prove that problem \eqref{eq:main-equation} has a unique solution, provided that the initial condition $u_0$ is bounded. Before, we need to establish the following maximum principle result. 
\begin{lemma}\label{lem:max-principle}
	Let $u = u(x, t)$  be a weak solution of the problem \eqref{eq:main-equation}, \ie satisfies  \eqref{eq:def-weak-solution} with  $u_0\in L^2(\Omega)\cap L^\infty(\Omega)$ then  $\|u\|_{L^\infty(0,T;L^\infty(\Omega))}\leq \|u_0\|_{ L^\infty(\Omega)}$ on $\Omega_T,$ i.e.,   $|u|\leq \|u_0\|_{ L^\infty(\Omega)}$ a.e. on $\Omega_T.$
\end{lemma}
\begin{proof}
	Set  $\Lambda=\|u_0\|_{L^\infty(\Omega)}$ and consider the convex  function $F:\R\to [0, \infty)$ defined by 
	\begin{align*}
	F(\tau) = 
	\begin{cases}
	\big(\tau +\Lambda\big)^2& \text{if}~~\tau <-\Lambda\\
	0 & \text{if}~~|\tau| \leq \Lambda\\
	\big(\tau -\Lambda\big)^2& \text{if}~~\tau >\Lambda.
	\end{cases}
	\end{align*}
	So that, $F(\tau)= 0$ if and only if $|\tau|\leq \Lambda$, in particular $F(u_0)=0$ a.e. on $\Omega$. By convexity, $F'$ is non-decreasing, i.e.,  $(F'(\tau_1)-F'(\tau_2))(\tau_1-\tau_2)\geq 0$  for all $\tau_1,\tau_2\in \R$ in particular,  since $F'(0)=0$, we have $F'(\tau_1)\tau_1\geq 0$  for all $\tau_1\in \R$. Furthermore,  $F'(u(\cdot, t)) \in \mathbb{X}_\nu(\Omega|\R^N)$ because $u(\cdot, t) \in \mathbb{X}_\nu(\Omega|\R^N)$ and one can check that $F'$ is Lipschitz since $F''$ is bounded.  Therefore, testing the equation \eqref{eq:def-weak-solution} against $\zeta= F'(u)$ gives 
	\begin{align*}
	\frac{d}{\d t}\int_\Omega F(u(x,t))\d x= -\mathcal{E}(u, F'(u))-\int_\Omega\varphi(v)(x) F'(u(x,t))u(x,t)\d x\leq 0. 
	\end{align*}
	Since $F(u_0)=0$ almost everywhere on $\Omega$, integrating the inequality gives  
	\begin{align*}
	\int_\Omega F(u(x,t))\d x \leq 0\quad\text{for all }\quad 0\leq t\leq T. 
	\end{align*}
	Thus, $F(u(x,t)) =0$ a.e. on $\Omega_T$, and hence  $|u(x,t)|\leq \Lambda$ a.e.  on $\Omega_T$. 
\end{proof}

\begin{theorem}\label{thm:uniqueness}
	Assume that $\varphi$ satisfies  Assumption~\ref{eq:assumption-phi}., $u_0\in L^2(\Omega)\cap L^\infty(\Omega)$  and that for $\Lambda= \|u_0\|_{L^\infty(\Omega)}$  we have  $|\varphi'(\tau)|\le \kappa$ for $\tau\in [-\Lambda T,\Lambda T]$ for some constant $\kappa>0$. Then the weak solution of problem \eqref{eq:main-equation}
	is unique provided that $\kappa \Lambda T^2<1$.
\end{theorem}
\begin{proof}
	Suppose that problem \eqref{eq:main-equation} has two weak solutions $u_1$ and $u_2$, and put $\displaystyle v_i(x)=\int_0^T u_i(x,t)\,\dt$, $i=1,2$. Then $u=u_1-u_2$ is a weak solution to
	
	\[
	\begin{cases}
	\partial_t u +\mathscr{L} u + \varphi (v_1)\, u_1 - \varphi(v_2)\,u_2 = 0 &\mbox{ in }\; \Omega_T,\\
	u = 0 &\mbox{ in }\; \Sigma , \\
	u(\cdot,0) = 0 &\mbox{ in }\; \Omega.
	\end{cases}
	\]
	The maximum principle in Lemma \ref{lem:max-principle}, implies that  $|u_i|\le \Lambda$  a.e. in $\Omega_T$ and hence $|v_i|\le \Lambda T$, $i=1,2$, a.e.  in $\Omega$.  Testing the above equation with  $u$ leads to the following equality:
	\[
	\frac{1}{2}\,\frac{\mathrm{d}}{\dt}\|u(\cdot, t)\|^2_{L^2(\Omega)}+\mathcal{E}(u,u)+\int_\Omega \varphi(v_1)\, u^2\,\dx
	+\int_\Omega \big(\varphi(v_1)-\varphi(v_2)\big)\,u_2\, u\,\dx =0
	\]
	which implies that
	\[
	\frac{1}{2}\,\|u(\cdot,t)\|^2_{L^2(\Omega)}+\int_0^t \mathcal{E}(u,u)\mathrm{d}\tau \le
	\kappa \Lambda \int_0^t\int_\Omega |v(x)|\,|u(x,\tau)|\,\dx\mathrm{d}\tau
	\]
	for all $t\in[0,T]$, where $v=v_1-v_2=\int_0^T u(\cdot, \tau)\d \tau .$
	Noticing that, 
	\begin{align*}
	\|v\|^2_{L^2(\Omega)}=  \int_\Omega\Big| \int_0^T u(x ,\tau)\d \tau\Big|^2\d x\leq T \int_0^T \|u(\cdot, \tau)\|^2_{L^2(\Omega)}\d \tau, 
	\end{align*}
	we get 
	\begin{align*}
	\int_0^t\int_\Omega |v(x)|\,|u(x,\tau)|\,\dx\mathrm{d}\tau 
	&\leq T\Big(\int_0^T\|u(\cdot, \tau)\|^2_{L^2(\Omega)}\,\mathrm{d}\tau\Big)^{1/2}\,\Big(\int_0^t\|u(\cdot, \tau)\|^2_{L^2(\Omega)}\,\mathrm{d}\tau\Big)^{1/2},
	\end{align*}
	Therefore, we obtain the following inequality for all $t\in[0,T]$
	\begin{align}\label{eq:gronwall-inequality}
	\displaystyle
	\|u(\cdot,t)\|^2_{L^2(\Omega)}\le 2\kappa \Lambda T \Big(\int_0^T\|u(\cdot, \tau)\|^2_{L^2(\Omega)}\,\mathrm{d}\tau\Big)^{1/2}\Big(\int_0^t\|u(\cdot, \tau)\|^2_{L^2(\Omega)}\,\mathrm{d}\tau\Big)^{1/2}. 
	\end{align}
	In short we rewrite the above inequality as follows  
	\begin{align*}
	\varrho'(t)\le 2\kappa \Lambda T\varrho^{1/2}(T)\, \varrho^{1/2}(t)\quad\text{with}\quad \varrho(t)=\int_0^t\|u(\cdot, \tau)\|^2_{L^2(\Omega)}\,\mathrm{d}\tau.
	\end{align*}
	A routine integration yields that $\varrho^{1/2}(t)\le\kappa \Lambda  T^2\varrho^{1/2}(T)$ and, in particular, $\varrho^{1/2}(T)\le\kappa \Lambda  T^2\varrho^{1/2}(T)$.  The latter inequality holds true only if $\varrho(T)=0$ since  $\kappa \Lambda T^2<1$,  which implies that $u=0$. 
\end{proof}

We now point out the following the closing remark which shows how the function spaces considered in this note extends our studies to a sightly different type of problems. 

\begin{remark}\label{rmk:Neumann-problem}
	Analogous results to  those obtained in this notes can be established replacing the Dirichlet complement condition $u = 0 \mbox{ in }\; (\Omc)\times (0,T),$  the problem  \ref{eq:main-equation} with  the  Neumann complement condition $\mathcal{N}u = 0 \mbox{ in }\; (\Omc)\times (0,T),$
	where $\mathcal{N}u$ represents the nonlocal normal derivative of $u$ across as defined in \eqref{eq:nonlocal-normal-derivative}.  To this end, it is decisive to taking into account the setting of Theorem \ref{thm:compactness}, namely that $\Omega$ is bounded and Lipschitz and that $\nu$ satisfies the asymptotic condition 
	\eqref{eq:asymp}, in such a way that the compactness of the embedding $V_{\nu}(\Omega|\R^N) \hookrightarrow L^2(\Omega)$ holds true. Wherefrom, one readily obtains (see \cite{Foghem,FK22}) the Poincar\'e type inequality 
	\begin{align*}
	\|u\|^2_{L^2(\Omega)} \leq C\mathcal{E}(u,u) \qquad\text{for all $u\in V_{\nu}(\Omega|\R^N)^\perp$}, 
	\end{align*}
	for some constant $C>0$ and where $V_{\nu}(\Omega|\R^N)^\perp = \big\{ V_{\nu}(\Omega|\R^N): \int_\Omega u \d x=0\big\}$.  These observations, alongside of our procedure,  allow to replace the space $\mathbb{X}_{\nu}(\Omega|\R^N) $ with the space $V_{\nu}(\Omega|\R^N)^\perp.$
\end{remark}

\vspace{2mm}
\noindent \textbf{Data Availability Statement (DAS)}: Data sharing not applicable, no datasets were generated or analysed during the current study.

\bibliographystyle{alpha}

\end{document}